\documentclass{amsart}
\usepackage{amsfonts}
\usepackage{amssymb}
\usepackage{times}

\newtheorem{pr}{Proposition}

\newtheorem{lemma}{Lemma}
\newtheorem{de}{Definition}
\newtheorem{teo}{Theorem}

\newfont{\hueca}{msbm10}

\begin{document}
\date{}
\author{A.J. Calder\'{o}n, L.M. Camacho, B.A. Omirov}

\title[Leibniz algebras of Heisenberg type]{\bf Leibniz algebras of Heisenberg type}

\maketitle
\begin{abstract}
We introduce and
  provide a classification theorem for the class of Heisenberg-Fock  Leibniz algebras. This category of algebras is formed
  by
those Leibniz algebras  $L$ whose corresponding Lie algebras are
Heisenberg algebras $H_n$ and whose  ${H_n}$-modules $I$, where
$I$ denotes the ideal generated by the squares of elements of $L$,
are isomorphic to Fock modules. We also consider the
three-dimensional Heisenberg algebra $H_3$
 and study three classes of Leibniz algebras with $H_3$ as corresponding Lie algebra,  by taking certain generalizations of the Fock module. Moreover, we describe the class of Leibniz algebras
 with $H_n$ as corresponding Lie algebra and such that the action $I \times H_n \to I$ gives rise to a minimal faithful representation of $H_n$. The classification of this family of Leibniz algebras
  for the case of $n=3$ is given.
\end{abstract}

\medskip \textbf{AMS Subject Classifications (2010):
17A32, 17B30, 17B10.}

\textbf{Key words:}  Heisenberg algebra, Leibniz algebra, Fock
representation, minimal faithful representation.

\section{Introduction}

The term {\it Leibniz algebra} was introduced in the  study of  a
non-antisymmetric analogue of  Lie algebras by Loday \cite{Loday},
being so the class of Leibniz algebras an extension of the one of
Lie algebras. However this kind of algebras was previously studied
under the name of $D$-algebras by D. Bloh \cite{B1,B2,B3}. Since
the  1993 Loday's work many researchers have been attracted by
this category  of algebras,  being remarkable the great activity
in this field developed in the last years. This activity has been
mainly focussed in the frameworks of low dimensional algebras,
nilpotence and physics applications  (see \cite{  Dzhu, Ayu01,p5,
 Cabezas, p3, ca2,  ca5, p2, le12, p1, p6,  p4,  le21,     ma1,
ma2}).

\begin{de}\rm
A {\it Leibniz algebra}  ${L}$ is a linear space over a base field
$\mathbb{F}$ endowed with a bilinear product $[\cdot, \cdot]$
satisfying the {\it  Leibniz identity}
$$ [[y,z],x]=[[y,x],z]+[y,[z,x]],$$ for all $x, y, z \in { L}$.
\end{de}

In presence of anti-commutativity, Jacobi identity becomes Leibniz
identity and therefore Lie algebras are examples of Leibniz algebras.
Throughout this paper  $\mathbb{F}$ will be  algebraically closed
and with zero characteristic.

\medskip

Let ${ L}$ be a Leibniz algebra. The ideal $I$ generated by the
squares of elements of the algebra $L$, that is  $I$ is generated by the set $\{[x,x]: x\in
{L}\}$, plays an important role in the theory since it determines
the (possible) non-Lie character of ${ L}$. From the Leibniz
identity, this ideal satisfies
$$[{L},I]=0.$$
The quotient algebra $L / I$ is a Lie algebra, called the {\it
corresponding Lie algebra} of $L$, and the map $I \times L / I \to
I$, $(i,[x]) \mapsto [i,x]$, endows $I$ of a structure of $L /
I$-module (see \cite{huel1, huel3}). Observe that we
  can
write
\begin{equation}\label{atleti}
L=V\oplus I
\end{equation}
 where $V$ is a linear complement of $I$ in $L$ and
  $V$  is isomorphic as linear space to  $L / {I}$. From here,
Leibniz algebras give us the opportunity of treating in an
unifying way a Lie algebra together with a module over it.

\bigskip

On the other hand, we recall that Heisenberg (Lie) algebras play
an important role in mathematical physics and geometry, in
particular in Quantum Mechanics (see for instance \cite{f1,
Bender, Berrada, f3, deJeu, Gelou,  chinos, Kaplain, Konstantina,
f6, f5, f2,  f4, Semmes, Thurston}). Indeed, the Heisenberg
Principle of Uncertainty implies the non-compatibility of position
and momentum observables acting on fermions. This
non-compatibility reduces to non-commutativity of the
corresponding operators. If we represent by $\overline{{x}}$ the
operator associated to position and by
${\frac{\overline{\partial}}{\partial x}}$ the one associated to
momentum (acting for instance on a space $V$ of differentiable
functions of a single variable), then $[\overline{x},{
\frac{\overline{\partial}}{\partial x}}]=\overline{1}_V$ which is
non-zero. Thus we can identify the subalgebra generated by
$\overline{1},\overline{{x}}$ and
${\frac{\overline{\partial}}{\partial x}}$ with the
three-dimensional Heisenberg  algebra whose multiplication table
in the basis
$\{\overline{1},\overline{x},\frac{\overline{\partial}}{\partial
x}\}$ has as unique non-zero product
$[\overline{x},\frac{\overline{\partial}}{\partial
x}]=\overline{1}$.

\medskip

 For any non-negative integer $k$ the {\it Heisenberg algebra} of
dimension $n=2k+1$ (denoted further by $H_n$) is characterized by the existence of a basis
\begin{equation}\label{base1}
 B =
\{\overline{1}, \overline{x}_1,\frac {\overline{\delta}}{\delta
x_1},\dots,\overline{x}_k,\frac{\overline{\delta}}{\delta x_k}\}
\end{equation}
  in which the
multiplicative non-zero relations are $$[\overline{x}_i, \frac
{\overline{\delta}}{\delta x_i}] = -[\frac
{\overline{\delta}}{\delta x_i}, \overline{x}_i]=\overline{1}$$
for $1 \leq i \leq k$.

\bigskip

In the present paper we are focusing in introducing  and studying
 several  classes of Leibniz algebras whose corresponding Lie
algebras are  Heisenberg algebras  $H_n$.  Recall that there is a
unique irreducible representation of the Heisenberg algebra (at
least a unique one that can be exponentiated). This is why
physicists are able to use the Heisenberg  commutation relations
to do calculations, without worry about what they are being
represented on. This representation is called the Fock (or
Bargmann-Fock) representation (see \cite{h1, h2, h3, h4, h5, h6}).
Physically this representation corresponds to an harmonic
oscillator, with the vector $\overline{1} \in \mathbb{C}[x]$ as
the vacuum state and $\overline{x}$ the operator that adds one
quantum to the vacuum state. This representation is also sometimes
known as the oscillator representation. For a given Heisenberg
algebra $H_n$, $n=2k+1$, this representation gives rise to the
so-called {\it Fock module} on $H_n$,  the linear space
${\mathbb{F}}[x_1,...,x_k]$ with the action induced by

\medskip
\begin{equation}\label{Fo}
\begin{array}{lll} ( p(x_1,...,x_k),\overline{1})& \mapsto &
p(x_1,...,x_k)\\{} ( p(x_1,...,x_k),\overline{x}_i)& \mapsto &
x_ip(x_1,...,x_k)\\{}
(p(x_1,...,x_k),\frac{\overline{\delta}}{\delta x_i})& \mapsto &
\frac {\delta}{\delta x_i}(p(x_1,...,x_k))
\end{array}
\end{equation}

 for any $p(x_1,...,x_k) \in
\mathbb{F}[x_1,...,x_k]$ and $i=1,...,k$.

\medskip

Taking now into account the above comments, we introduce in
Section 2  the class of {\it Heisenberg-Fock  Leibniz algebras} as
those Leibniz algebras   whose corresponding Lie algebras are
 $H_n$ and whose  ${H_n}$-modules $I$
are isomorphic to Fock modules, and provide a classification
theorem.  Thus, we have the opportunity of  considering Heisenberg
Lie algebras together with their Fock representations in a
unifying viewpoint. In this section we also consider a
generalization of this class of algebras by means of a direct sum
of Heisenberg algebras as corresponding Lie algebras,  and provide
also a classification theorem.

In Section 3, we center in the three-dimensional Heisenberg algebra $H_3$ and study three classes of Leibniz algebras with
$H_3$ as corresponding Lie algebra by taking certain generalizations of the Fock module. We also note that Sections 2
and 3 allow us to introduce several new classes of infinite-dimensional Leibniz algebras.

Finally, in Section 4, we deal with the category of Leibniz algebras with $H_n$ as corresponding algebra and such that the
action $I \times H_n \to I$ gives rise to a minimal faithful representation of $H_n$. A description of this category of
algebras is given and also a classification theorem when $n=3$.

\section{Classification of Heisenberg-Fock type  Leibniz algebras}

\subsection{Classification of $HFL_n$}

Consider a Heisenberg algebra $H_n$, with $n=2k+1$, and its Fock
module ${\mathbb{F}}[x_1,...,x_k]$ under the action (\ref{Fo}).
The {\it Heisenberg-Fock  Leibniz algebra} $HFL_n$ is defined as
the Leibniz algebra with corresponding Lie algebra $H_n$ and such
that the action $I \times H_n \to I$ makes of $I$ the Fock module.
Since ${\mathbb{F}}[x_1,...,x_k]$ is infinite-dimensional we get a
family of infinite-dimensional Leibniz algebras.

\begin{teo}\label{1} The  Heisenberg-Fock  Leibniz algebra  $HFL_{n}$
admits a basis
$$
 \{\overline{1}, \overline{x}_i,\frac {\overline{\delta}}{\delta
x_i}, \  x_1^{t_1}x_2^{t_2}\dots x_k^{t_k} \ | \ t_i\in \mathbb{N}\cup \{0\}, \ 1\leq i \leq k\}
$$
in such a way that the multiplication table on this basis has
the form:
$$\begin{array}{ll}
[\overline{x_i},\frac {\overline{\delta}}{\delta
x_i}]=\overline{1},&1 \leq i \leq k,\\{}
[\frac
{\overline{\delta}}{\delta x_i},\overline{x_i}]=-\overline{1},& 1
\leq i \leq k,
\end{array}$$
$$\begin{array}{ll}
[x_1^{t_1}x_2^{t_2}\dots
x_k^{t_k},\overline{1}] = x_1^{t_1}x_2^{t_2}\dots x_k^{t_k},&\\{}
[x_1^{t_1}x_2^{t_2}\dots x_k^{t_k},\overline{x}_i] =
x_1^{t_1}\dots x_{i-1}^{t_{i-1}}x_i^{t_i+1} x_{i+1}^{t_{i+1}}\dots
x_k^{t_k},&1 \leq i \leq k,\\{} [x_1^{t_1}x_2^{t_2}\dots
x_k^{t_k},\frac {\overline{\delta}}{\delta x_i}] = t_i
x_1^{t_1}\dots x_{i-1}^{t_{i-1}} x_i^{t_i-1}x_{i+1}^{t_{i+1}}\dots
x_k^{t_k},&1 \leq i \leq k,
\end{array}$$
where the omitted products are equal to zero.
\end{teo}

\begin{proof}
Taking into account  Equations (\ref{atleti}) and  (\ref{Fo}) we
conclude that
$$ \{\overline{1}, \overline{x}_i,\frac {\overline{\delta}}{\delta
x_i}, \  x_1^{t_1}x_2^{t_2}\dots x_k^{t_k} \ | \ t_i\in
\mathbb{N}\cup \{0\}, \ 1\leq i \leq k\}$$
is a basis of $HFL_n$ and
$$\begin{array}{l}
[x_1^{t_1}x_2^{t_2}\dots x_k^{t_k},\overline{1}] =
x_1^{t_1}x_2^{t_2}\dots x_k^{t_k},\\{} [x_1^{t_1}x_2^{t_2}\dots
x_k^{t_k},\overline{x}_i] = x_1^{t_1}\dots
x_{i-1}^{t_{i-1}}x_i^{t_i+1} x_{i+1}^{t_{i+1}}\dots x_k^{t_k},\\{}
[x_1^{t_1}x_2^{t_2}\dots x_k^{t_k},\frac
{\overline{\delta}}{\delta x_i}] = t_i x_1^{t_1}\dots
x_{i-1}^{t_{i-1}} x_i^{t_i-1}x_{i+1}^{t_{i+1}}\dots x_k^{t_k},
\end{array}$$
for $1\leq i\leq k.$

Observe that we can write
$$\begin{array}{ll}
[\overline{x_i},\overline{1}]=p_i(x_1, x_2, \dots, x_k), & 1 \leq
i \leq k,\\{} [\frac {\overline{\delta}}{\delta
x_i},\overline{1}]=q_i(x_1, x_2, \dots, x_k),&1 \leq i \leq k,\\{}
[\overline{1},\overline{1}]=r(x_1, x_2, \dots, x_k),&
\end{array}$$
where $p_i, q_i, r \in {\mathbb{F}}[x_1,...,x_k]$.

Taking the following change of basis,
$$\overline{x_i}^{\prime} = \overline{x_i} - p_i(x_1, x_2, \dots, x_k),\quad 1 \leq i \leq k,$$
$$\frac {\overline{\delta}}{\delta x_i}^{\prime} = \frac {\overline{\delta}}{\delta x_i} - q_i(x_1,
x_2, \dots, x_k),\quad 1 \leq i \leq k,$$
$$\overline{1}^{\prime} = \overline{1} - r(x_1,x_2, \dots, x_k),$$
we derive
$$[\overline{x_i},\overline{1}]=0, \quad [\frac {\overline{\delta}}{\delta x_i},\overline{1}]=0, \quad
[\overline{1},\overline{1}]=0, \quad 1 \leq i \leq k.$$

Now denote
$$\begin{array}{lll}
[\overline{x_i},\overline{x_j}]=a_{i,j}(x_1, x_2, \dots, x_k),&[\frac {\overline{\delta}}{\delta x_i},\frac{\overline{\delta}}{\delta x_j}]=
b_{i,j}(x_1, x_2, \dots, x_k),&  1 \leq i, j \leq k,\\{}
[\frac {\overline{\delta}}{\delta x_i},\overline{x_j}]=
c_{i,j}(x_1, x_2, \dots, x_k),&[\overline{x_i}, \frac {\overline{\delta}}{\delta x_j}]=
d_{i,j}(x_1, x_2, \dots, x_k), &1 \leq i,j \leq k, \ i \neq j,\\{}
[\overline{x_i}, \frac {\overline{\delta}}{\delta x_i}]=
\overline{1} + e_{i}(x_1, x_2, \dots, x_k),&[\frac {\overline{\delta}}{\delta x_i}, \overline{x_i}]=-\overline{1} + f_{i}(x_1, x_2, \dots, x_k), & 1 \leq i \leq k,\\{}
[\overline{1}, \overline{x_i}]= h_{i}(x_1, x_2, \dots, x_k),&[\overline{1}, \frac {\overline{\delta}}{\delta x_i}]= g_{i}(x_1, x_2, \dots, x_k), & 1 \leq i \leq k.
\end{array}$$

The Leibniz identity on the following triples imposes further constraints on the products.
$$\begin{array}{llrl}
\mbox{ Leibniz identity }& & \mbox{ Constraint }&\\[1mm]
\hline \hline\\
\{\overline{x_i},\overline{x_j},\overline{1}\}&\Rightarrow &a_{i,j}(x_1, x_2,
\dots, x_k)=0, & 1 \leq i,j \leq k,\\[1mm]
\{\frac {\overline{\delta}}{\delta x_i},
\frac{\overline{\delta}}{\delta x_j},\overline{1}\}&\Rightarrow&b_{i,j}(x_1, x_2, \dots, x_k)=0, & 1 \leq i,j \leq k,\\[1mm]
\{\frac {\overline{\delta}}{\delta
x_i}, \overline{x_j},\overline{1}\}&\Rightarrow&c_{i,j}(x_1, x_2, \dots, x_k)=0, & 1 \leq i,j \leq k,\ i\neq j ,\\[1mm]
\{\overline{x_i},\frac{\overline{\delta}}{\delta x_j},\overline{1}\}&\Rightarrow&d_{i,j}(x_1, x_2, \dots, x_k)=0,& 1 \leq i,j \leq k,\ i\neq j,\\[1mm]
\{\overline{x_i},\frac{\overline{\delta}}{\delta
x_i},\overline{1}\}&\Rightarrow&e_{i}(x_1, x_2, \dots, x_k)=0,&1\leq i\leq k,\\[1mm]
\{\frac{\overline{\delta}}{\delta x_i},
\overline{x_i},\overline{1}\}&\Rightarrow &f_{i}(x_1, x_2, \dots, x_k)=0,&1\leq i\leq k,\\[1mm]
\{\overline{1}, \overline{x_i},\overline{1}\}&\Rightarrow&h_{i}(x_1, x_2, \dots, x_k)= 0,&1\leq i\leq k,\\[1mm]
\{\overline{1}, \frac{\overline{\delta}}{\delta x_i},\overline{1}\}&\Rightarrow &g_{i}(x_1, x_2, \dots, x_k)=0,&1\leq i\leq k.
\end{array}$$

The proof is complete.
\end{proof}


\smallskip


\subsection{Classification of generalized Heisenberg-Fock Leibniz algebras}
 In this subsection  we are interested in classifying
the class of (infinite-dimensional) Leibniz algebras formed by
those Leibniz algebras $L$ satisfying that their corresponding Lie
algebras are finite direct sums of Heisenberg algebras  and  that
the actions on $I $ are induced by Fock representations.

Since
\begin{equation}\label{me1}
L/I \cong H_{2k_1+1}\oplus H_{2k_2+1}\oplus H_{2k_3+1}\oplus \dots
\oplus H_{2k_s+1},
\end{equation}
we easily get
\begin{equation}\label{eq1}
{\mathcal B}_i:=\{\overline{1_{i}},
\overline{x_{1,i}},\overline{x_{2,i}},...,\overline{x_{k_i,i}},\frac
{\overline{\delta}}{\delta x_{1,i}},\frac
{\overline{\delta}}{\delta x_{2,i}},...,\frac
{\overline{\delta}}{\delta x_{k_i,i}} \}
\end{equation}
for the standard basis of
$H_{2k_i+1}$, $i \in \{1,2,...,s\}$.

We put
\begin{equation}\label{me2}
I={\mathbb F}[x_1, ..., x_n],
\end{equation}
where $n=k_1+k_2+ \cdots +k_s$.

The action $$ I \times L/I \to I$$
given by
$$\begin{array}{lll}
(p(x_1,\dots,x_n),\overline{1_{i}})& \mapsto &
p(x_1,\dots,x_n)\\{} (p(x_1,\dots,x_n),\overline{x_{j,i}})&
\mapsto & p(x_1,\dots,x_n)x_{k_1+k_2+\cdots+k_{i-1}+j}\\{}
(p(x_1,\dots,x_n),\frac{\overline{\delta}}{\delta x_{j,i}})&
\mapsto &\frac {{\delta}}{\delta
x_{k_1+k_2+\cdots+k_{i-1}+j}}p(x_1,\dots,x_n)
\end{array}$$ for any $p(x_1,\dots,x_n) \in {\mathbb F}[x_1, ..., x_n]$ and
$(i,j)$ with  $i \in \{1,2,...,s\}$, $j\in \{1,...,k_i\}$,  endows
$I$ of a structure of $L/I$-module. Hence, we get a new family
of Heisenberg-Fock type Leibniz algebras which generalize the
previous ones considered in  $\S2.1$ (case $s=1$), that we
call {\it generalized Heisenberg-Fock  Leibniz algebras}, by
introducing the  algebras $L=L/I \oplus I$ with $L/I$ and $I$ as
in Equations (\ref{me1}) and (\ref{me2}). We will denote them as
$$HFL_{2k_1+1,2k_2+1,...,2k_s+1}.$$

\medskip

Our aim is to classify  this class of Leibniz algebras.

\medskip

By taking into account the previous arguments, it
is clear that for any $i\in \{1,2,...,s\}$ we have
$[H_{2k_i+1},H_{2k_i+1}] \subset H_{2k_i+1}$ being the
multiplication table among the elements in the basis ${\mathcal
B}_i$ as in Theorem \ref{1}. Therefore, we only need to study the
products $[H_{2k_i+1},H_{2k_j+1}] $ with $i,j\in \{1,2,...,s\}$
and $i \neq j$.

\begin{lemma}\label{me3}
Let $a \in {\mathcal B}_i$ and $b \in {\mathcal B}_j$, $i,j\in
\{1,2,...,s\}$ with  $i \neq j$. Then $[a,b]=0$.
\end{lemma}
\begin{proof}
For $i \neq j$ we have  $[a,b]=p$ and $[b,\overline{1_{i}}]=q$ for
some $p, q \in {\mathbb F}[x_1, ..., x_n]$.  Taking now into
account Theorem \ref{1} we  derive $[a,\overline{1_{i}}]=0$ and so
$$p=[[a,b],\overline{1_{i}}]=[[a,\overline{1_{i}}],b]+[a,[b,\overline{1_{i}}]]=0.$$
\end{proof}

The next theorem is now consequence of Theorem \ref{1} and Lemma
\ref{me3}.

\begin{teo} The Leibniz algebra  $HFL_{2k_1+1,2k_2+1,...,2k_s+1}$
admits a basis (see Equations (\ref{eq1}) and (\ref{me2}))
$$
{\mathcal B}_1 \dot{\cup} {\mathcal B}_2 \dot{\cup} \cdots
\dot{\cup} {\mathcal B}_s \dot{\cup} \{x_1^{t_1}x_2^{t_2}\cdots
x_n^{t_n}\ | \ t_i\in \mathbb{N}\cup \{0\}, \ 1\leq i \leq n\},$$
where $n=k_1+k_2+\cdots+k_s,$ and in such a way that the
multiplication table on this basis has the form:
$$\begin{array}{lll}
[\overline{x_{j,i}},\frac {\overline{\delta}}{\delta
x_{j,i}}]=\overline{1_i}, \qquad \quad [\frac
{\overline{\delta}}{\delta
x_{j,i}},\overline{x_{j,i}}]=-\overline{1_i},& & \\[1mm]
[x_1^{t_1}x_2^{t_2}\dots x_n^{t_n},\overline{1_i}] = x_1^{t_1}x_2^{t_2}\dots x_n^{t_n},&&\\[1mm]
[x_1^{t_1}x_2^{t_2}\dots x_n^{t_n},\overline{x}_{j,i}] =
x_1^{t_1}\dots
x_{k_1+\cdots+k_{i-1}+j-1}^{t_{k_1+\cdots+k_{i-1}+j-1}}x_{k_1+\cdots+k_{i-1}+j}^{t_{k_1+\cdots+k_{i-1}+j}+1}
x_{k_1+\cdots+k_{i-1}+j+1}^{t_{k_1+\cdots+k_{i-1}+j+1}}\dots x_n^{t_n},&&\\[1mm]
[x_1^{t_1}x_2^{t_2}\dots x_k^{t_k},\frac
{\overline{\delta}}{\delta x_{j,i}}] = t_{k_1+\cdots+k_{i-1}+j}
x_1^{t_1}\dots
x_{k_1+\cdots+k_{i-1}+j-1}^{t_{k_1+\cdots+k_{i-1}+j-1}}
x_{k_1+\cdots+k_{i-1}+j}^{t_{k_1+\cdots+k_{i-1}+j}-1}x_{k_1+\cdots+k_{i-1}+j+1}^{t_{k_1+\cdots+k_{i-1}+j+1}}\dots
x_n^{t_n},&&
\end{array}$$
for  $1 \leq i \leq s,$  $1 \leq j \leq k_i$ and where the omitted
products are equal to zero.
\end{teo}

\bigskip

\section{Several degenerations of the Fock representation for the \\ $3$-dimensional Heisenberg algebra}

In this section we consider several degenerations of the Fock
representation of the Heisenberg algebra   $H_3$. First, we study
when an extension of the Fock action  ${\mathbb F}[x] \times H_3
\to {\mathbb F}[x]$, (see Equation (\ref{Fo})),  by allowing
arbitrary polynomials as results of the action of a fixed element
in the basis $\{\overline{1},\overline{x},
\frac{\overline{\delta}}{\delta x}\}$ of  $H_3$ over the elements
of ${\mathbb F}[x]$, makes of ${\mathbb F} [x]$ an $H_3$-module.
Second, the new $H_3$-modules obtained in this way give rise to
new classes of Leibniz algebras that will be described.

\medskip

For any linear mapping $\Omega:\mathbb{F}[x] \to \mathbb{F}[x]$,
consider the linear space $\mathbb{F}[x]$ with the action induced
by the following applications:
 $$\begin{array}{ll}
 \begin{array}{llll}
\psi_1:&\mathbb{F}[x] \times H_3 &\rightarrow & \mathbb{F}[x]\\[1mm]

&(p(x),\overline{1})&\mapsto & \Omega(p(x))\\
&(p(x),\overline{x})&\mapsto & xp(x)\\
&(p(x),\frac{\overline{\delta}}{\delta x})& \mapsto &
\frac{\delta}{\delta x} p(x).
\end{array}&
\begin{array}{llll}
\psi_2:&\mathbb{F}[x] \times H_3 &\rightarrow & \mathbb{F}[x]\\[1mm]

&(p(x),\overline{1})&\mapsto & p(x)\\
&(p(x),\overline{x})&\mapsto & \Omega(p(x))\\
&(p(x),\overline{\frac{\delta}{\delta x}})& \mapsto &
\frac{\delta}{\delta x} p(x).
\end{array}\\[15mm]
\begin{array}{llll}
\psi_3:&\mathbb{F}[x] \times H_3 &\rightarrow & \mathbb{F}[x]\\[1mm]

&(p(x),\overline{1})&\mapsto & p(x)\\
&(p(x),\overline{x})&\mapsto & xp(x)\\
&(p(x),{\frac{\overline{\delta}}{\delta x}})& \mapsto &
\Omega(p(x))
\end{array}
\end{array}$$
for any  $p(x)\in \mathbb{F}[x].$

From now on, let us denote by
 $\{x^i\}_{i\in \mathbb{N}\cup \{0\}}$   the canonical
basis of $\mathbb{F}[x].$  By considering   $\psi_1(p(x) ,
[\overline{x}, {\frac{\overline{\delta}}{\delta x}}])$, it is
immediate to get that the first action $\psi_1$ makes of
$\mathbb{F}[x]$ an $H_3$-module if and only if
$\Omega=1_{\mathbb{F}[x]}$. As  consequence we have.

\begin{pr} The Leibniz algebras obtained from the first action $\psi_1$ are the same as those obtained in Theorem \ref{1}.
\end{pr}
%


\medskip
Consider now the second action $\psi_2:\mathbb{F}[x] \times H_3
\rightarrow  \mathbb{F}[x]$.

\begin{pr}\label{77} The  action $\psi_2$ makes of $\mathbb{F}[x]$ an
$H_3$-module if and only if
\begin{equation}\label{eq6} \Omega(x^i) = x^{i+1} + \sum\limits_{k=0}^i
c_k\Big(\begin{array}{c} i
\\ k \end{array}\Big) x^{i-k},\end{equation}
where $\{c_k \}_{k \in {\mathbb N} \cup \{0\}}$ is a fixed
sequence in $\mathbb{F}$ and $\Big(\begin{array}{c} i
\\ k \end{array}\Big)$ are binomial coefficients.
\end{pr}
\begin{proof}
Suppose $\mathbb{F}[x]$ is an $H_3$-module through the action
$\psi_2$. Then we have
$$x^i = [x^i,\overline{1}]= [x^i,[\overline{x},\frac {\overline{\delta}}{\delta x}]]=
[[x^i,\overline{x}],\frac {\overline{\delta}}{\delta x}] -
[[x^i,\frac {\overline{\delta}}{\delta x}],\overline{x}] =
[[x^i,\overline{x}],\frac {\overline{\delta}}{\delta x}] -
[ix^{i-1},\overline{x}]$$ and so
\begin{equation}\label{eq7}
[[x^i,\overline{x}],\frac {\overline{\delta}}{\delta x}] =x^i+
[ix^{i-1},\overline{x}].\end{equation}

Taking into account  Equation \eqref{eq7},   we can easily prove
by induction
 \eqref{eq6}.
 Indeed,  for $i=0$ we get from \eqref{eq7} that
$ [[1,\overline{x}],\frac {\overline{\delta}}{\delta x}] =1,$
which implies $ [1,\overline{x}] = x+c_0 = \Omega(1).$ For $i=1$
the same equation allows us to get $[[x,\overline{x}],\frac
{\overline{\delta}}{\delta x}] =x+ [1,\overline{x}] = 2x+c_0$ and
so $ [x,\overline{x}] = x^2+c_0x+c_1= \Omega(x).$

Let the induction hypothesis true for $i=j$ and we will show it
for $i=j+1.$ Taking into account \eqref{eq7}  we have
$$\begin{array}{lll}
[[x^{j+1},\overline{x}],\frac {\overline{\delta}}{\delta x}] &=& x^{j+1}+
[(j+1)x^{j},\overline{x}] = x^{j+1}+ (j+1)(x^{j+1} +
\sum\limits_{k=0}^j c_k\Big(\begin{array}{c} j
\\ k \end{array}\Big) x^{j-k})=\\{}
&=&(j+2)x^{j+1} + \sum\limits_{k=0}^j c_k(j+1)\Big(\begin{array}{c} j
\\ k \end{array}\Big) x^{j-k}=\\{}
\end{array}$$
$$\begin{array}{lll}
&=&(j+2)x^{j+1} + \sum\limits_{k=0}^j c_k(j+1) \frac{j!}{k!(j-k)!}
x^{j-k}=\\{} &=& (j+2)x^{j+1} + \sum\limits_{k=0}^j c_k
\frac{(j+1)!}{k!(j+1-k)!} (j+1-k) x^{j-k}.
\end{array}$$

From here
$$[x^{j+1},\overline{x}] = x^{j+2} + \sum\limits_{k=0}^{j} c_k \frac{(j+1)!}{k!(j+1-k)!} x^{j+1-k} + c_{j+1} =
x^{j+2} + \sum\limits_{k=0}^{j+1} c_k \Big(\begin{array}{c} j+1
\\ k \end{array}\Big) x^{j+1-k},$$
that is, $$\Omega(x^{j+1}) = x^{j+2} + \sum\limits_{k=0}^{j+1} c_k
\Big(\begin{array}{c} j+1
\\ k \end{array}\Big) x^{j+1-k}.$$

The converse is of immediate verification.
\end{proof}

\begin{pr}\label{type2}
Any  Leibniz algebra obtained from the second action $\psi_2$
admits a basis
$$\{\overline{1}, \overline{x},\frac
{\overline{\delta}}{\delta x}\} \dot{\cup} \{x^i: i\in \mathbb{N}
\cup \{0\} \}$$ in such a way that the multiplication table on
this basis has the form:
$$\begin{array}{lllll}
[x^i,\overline{1}]=x^i,& &
[x^i,\overline{x}]=\Omega(x^i), & & [x^i,\frac
{\overline{\delta}}{\delta x}]=ix^{i-1},\\[1mm]
[\overline{x},\frac {\overline{\delta}}{\delta x}]=\overline{1},& & [\frac {\overline{\delta}}{\delta
x},\overline{x}]=-\overline{1},  &  &
\end{array}$$
where the omitted products are equal to zero and $\Omega(x^i) $
satisfies Equation (\ref{eq6}).
\end{pr}

\begin{proof}
By Proposition \ref{77} we have the restriction on $\Omega(x^i)$.
On the other hand, we know
$$\begin{array}{lllll}
[x^i,\overline{1}]=x^i,& &[x^i,\overline{x}]=\Omega(x^i),& &
[x^i,\frac {\overline{\delta}}{\delta x}]=ix^{i-1},\\{}

[\overline{x},\overline{1}]=p(x),& &[\overline{x},\frac
{\overline{\delta}}{\delta x}]=\overline{1}+q(x),& &
[\overline{x},\overline{x}]=a(x),\\{}

 [\frac
{\overline{\delta}}{\delta x},\overline{1}]=r(x),& &
 [\frac
{\overline{\delta}}{\delta x},\frac {\overline{\delta}}{\delta
x}]=b(x),& &[\frac{\overline{\delta}}{\delta
x},\overline{x}]=-\overline{1}+s(x),\\{}

[\overline{1},\overline{x}]=c(x),& &
[\overline{1},\overline{1}]=d(x),& &[\overline{1},\frac
{\overline{\delta}}{\delta x}]=e(x).
\end{array}$$

By making the  change of basis $\overline{1}'=\overline{1}+q(x)$
we can suppose that $[\overline{x},\frac
{\overline{\delta}}{\delta x}]=\overline{1}.$

Now, from Leibniz identity we obtain the following equations:
$$\begin{array}{lll}
\mbox{Leibniz identity }& &\mbox{Constraint}\\
\hline \hline\\
\{\overline{1},\overline{1},\overline{1}\}&\Rightarrow& c(x)=[d(x),\overline{x}],\\
\{\overline{1},\overline{1},\frac {\overline{\delta}}{\delta x}\}&\Rightarrow&e(x)=\frac {{\delta}}{\delta x}(d(x)),\\
\{\overline{1},\overline{x},\frac {\overline{\delta}}{\delta x}\}&\Rightarrow&[e(x),\overline{x}]=\frac{\delta}{\delta x}(c(x))-d(x),\\
\{\overline{x},\overline{1},\overline{x}\}&\Rightarrow& a(x)=[p(x),\overline{x}],\\
\{\overline{x},\overline{1},\frac {\overline{\delta}}{\delta x}\}&\Rightarrow& d(x)=\frac {{\delta}}{\delta x}(p(x)),\\
\{\overline{x},\overline{x},\frac {\overline{\delta}}{\delta x}\}&\Rightarrow&p(x)+c(x)=\frac {{\delta}}{\delta x}(a(x)),\\
\{\frac {\overline{\delta}}{\delta x},\overline{1},\overline{x}\}&\Rightarrow&s(x)=d(x)+[r(x),\overline{x}],\\
\{\frac {\overline{\delta}}{\delta x},\overline{1},\frac {\overline{\delta}}{\delta x}\}&\Rightarrow& b(x)=\frac {{\delta}}{\delta x}(a(x)),\\
\{\frac {\overline{\delta}}{\delta x},\overline{x},\frac
{\overline{\delta}}{\delta x}\}&\Rightarrow&
[b(x),\overline{x}]=-e(x)-r(x)+\frac {{\delta}}{\delta x}(a(x)).
\end{array}$$

By  making the next change of basis:
$$\begin{array}{l}
\overline{1}'=\overline{1}-\frac {{\delta}}{\delta x}(p(x)),\\
\overline{x}'=\overline{x}-p(x),\\
\frac {\overline{\delta}}{\delta x}'=\frac
{\overline{\delta}}{\delta x}-r(x),
\end{array}$$
we obtain the family of the proposition.
\end{proof}

Finally  we consider the third action  $\psi_3:\mathbb{F}[x]
\times H_3 \rightarrow  \mathbb{F}[x]$,  being then
$$\begin{array}{ll}
[x^i,\overline{1}]=x^i,&\\{} [x^i,\overline{x}]=x^{i+1},\\{}
[x^i,\frac {\overline{\delta}}{\delta x}]=\Omega(x^i),&i\in
\mathbb{N} \cup \{0\}.
\end{array}$$

By arguing in a similar way to Propositions \ref{77} and
\ref{type2} we can prove the next results.

\begin{pr} The  action $\psi_3$ makes of $\mathbb{F}[x]$ an
$H_3$-module if and only if
\begin{equation}\label{eq8}  \Omega(x^i) = ix^{i-1} + x^ic(x).\end{equation}
for a fixed $c(x) \in \mathbb{F}[x]$ and $i \in \mathbb{N} \cup
\{0\}$.
\end{pr}
%
%

\begin{pr}
Any  Leibniz algebra obtained from the third action $\psi_3$
admits a basis
$$\{\overline{1}, \overline{x},\frac
{\overline{\delta}}{\delta x}\} \dot{\cup} \{x^i: i\in \mathbb{N}
\cup \{0\} \}$$ in such a way that the multiplication table on
this basis has the form:
$$\begin{array}{lllll}
[x^i,\overline{1}]=x^i,& &
[x^i,\overline{x}]=x^{i+1}, & & [x^i,\frac
{\overline{\delta}}{\delta x}]=\Omega(x^i),\\[1mm]
 [\overline{x},\frac {\overline{\delta}}{\delta x}]=\overline{1},& &[\frac {\overline{\delta}}{\delta
x},\overline{x}]=-\overline{1},  & &
\end{array}$$
where the omitted products are equal to zero and $\Omega(x^i)$
satisfies  Equation (\ref{eq8}).
\end{pr}

\medskip

\section{Leibniz algebras of minimal faithful representation-Heisenberg type}

\subsection{General case.}
Let $H_{2m+1}$ be a Heisenberg algebra of dimension $2m+ 1,$ then
it is well-known that  its minimal faithful representations have
dimension $m + 2$ , (see \cite{Burde}). From now on, for a more
comfortable notation,  we will denote by $$\{x_1, x_2, \dots, x_m,
y_1, y_2, \dots, y_m, z\}$$ the standard basis of
 $H_{2m+1},$ (see Equation (\ref{base1})), where the non-zero products are
 $$[y_i,x_i]=-[x_i,y_i]=z.$$

 By \cite{Corwin}, we can take as   minimal faithful
representation the linear mapping  $$\varphi: H_{2m+1} \rightarrow
{\rm End}(I),$$ where $I$ is an $(m+2)$-dimensional linear space
with a fixed basis $\{e_1, e_2, \dots, e_{m+2}\}$, determined  by
$$\begin{array}{ll}
\varphi(x_i) = E_{1,i+1}& 1 \leq i \leq m,\\
\varphi(y_i) = E_{i+1,m+2}& 1 \leq i \leq m,\\
\varphi(z) = E_{1,m+2}.&
\end{array}$$
Here $E_{i,j}$ denotes the elemental matrix with $1$ in the
$(i,j)$ slot and $0$ in the remaining places and we have
$\varphi([x,y])(e)= \varphi(y)\big(\varphi(x) (e)\big) -
\varphi(x)\big(\varphi(y) (e)\big)$ for any  $x,y \in H_{2m+1}$
 and  $e \in I.$  Observe that $H_{2m+1}$
corresponds to the $(m+2) \times (m+2)$ matrices
$$\left(\begin{matrix}0&a_2&a_3&\dots&a_{m+1}&c\\0&0&0&\dots&0&b_{2}\\0&0&0&\dots&0&b_{3}\\
\vdots&\vdots&\vdots&\vdots&\vdots&\vdots\\0&0&0&\dots&0&b_{m+1}\\0&0&0&\dots&0&0\end{matrix}\right).$$

This representation makes of $I$ an $H_{2m+1}$-module under the
action

\begin{equation}\label{final}
\begin{array}{lllll}
\phi:&I \times H_{2m+1} & \to & I &\\
&  (e_{i+1},x_i) & \mapsto & e_{1}, & 1 \leq i \leq m,\\
&(e_{m+2}, y_i)& \mapsto &e_{i+1}, & 1 \leq i \leq m,\\
&(e_{m+2}, z )& \mapsto &e_1,&
\end{array}
\end{equation}
 being zero the remaining products among the bases elements in the action.

\smallskip

In this section we are going to study  the Leibniz algebras
$(L,[\cdot, \cdot])$ satisfying  that $L/ I \cong H_{2m+1}$ and
where the $H_{2m+1}$-module $I$ is isomorphic to the  minimal
faithful representation $(I, \phi)$. From the above,
 ${\rm dim} L = 3m+3$ and $\{ x_1, x_2, \dots, x_m, y_1, y_2, \dots, y_m, z,e_1, e_2,
\dots, e_{m+2}\}$ is  a basis of $L.$ We also have
$$\begin{array}{ll}
[e_{i+1}, x_i] = e_1, & 1 \leq i \leq m,\\{} [e_{m+2},y_i] =
e_{i+1}, & 1 \leq i \leq m,\\{} [e_{m+2}, z] = e_1.&
\end{array}$$

\begin{teo}\label{general}
Let $L$ be a Leibniz algebra such that $L/ I \cong H_{2m+1}$
($m\neq 1$) and $I$ is the $L/ I$-module with the
minimal faithful representation given by Equation (\ref{final}).
Then $L$ admits a basis
$$\{ x_1, x_2, \dots, x_m, y_1, y_2,
\dots, y_m, z,e_1, e_2, \dots, e_{m+2}\}$$ in such a way that the
multiplications table on this basis has the form
$$\begin{array}{ll}
[e_{i+1}, x_i] = e_1, & [e_{m+2},y_i] = e_{i+1},\\[2mm]
[e_{m+2}, z] = e_1,&[x_i,x_j]=\sum\limits_{s=1}^{m+1} \alpha_{i,j}^s e_s,\\[2mm]
[x_i,y_j]=\gamma_{i,j}  e_1, \ \ i\neq j,&[x_i,y_i]=-z+\delta_i  e_1+\tau  e_2+\sum\limits_{s=2}^{m} \nu_{1,s}^2 e_{s+1},\\[2mm]
[y_i,y_j]=\beta_{i,j}  e_1, & [y_1,x_1]=z,\\[2mm]
[y_i,x_j]=\sum\limits_{s=1}^{m+1} \nu_{i,j}^s e_s, \ \ i\neq j,&
[y_i,x_i]=z+(\nu_{i,1}^{i+1}-\tau )e_2+\varepsilon_i^{i+1}
e_{i+1}+\sum\limits_{s=2\, s\neq i}^{m}
(\nu_{i,s}^{i+1}-\nu_{1,s}^2) e_{s+1},
\ i\neq 1,\\[2mm]
[z,x_1]=  \tau  e_1, & [z,x_i]=  \nu_{1,i}^2 e_1,\ \  i\neq 1,
\end{array}$$
for  $1\leq i,j\leq m$,  where any $\alpha_{p,q}^r, \gamma_{p,q},
\delta_p, \tau, \nu_{p,q}^r , \beta_{p,q},
 \varepsilon_p^r \in
{\mathbb F}$ and where the omitted products are equal to zero.
\end{teo}
\begin{proof}
We consider the following products:
$$[y_i,x_i]=z+\sum\limits_{k=1}^{m+2} \varepsilon_{i}^k e_k, \ \ \ \ 1
\leq i\leq m.$$
Putting $z'=z+\sum\limits_{k=1}^{m+2}
\varepsilon_{1}^k e_k$ we can assume $[y_1,x_1]=z.$ Thus, we
have
$$\begin{array}{lll}
[e_{i+1}, x_i] = e_1, & [e_{m+2},y_i] = e_{i+1}, & [e_{m+2}, z] = e_1,\\[3mm]
[x_i,x_j]=\sum\limits_{k=1}^{m+2} \alpha_{i,j}^k e_k, &[x_i,y_j]=\sum\limits_{k=1}^{m+2} \gamma_{i,j}^k e_k,\ \ i\neq j &
[x_i,y_i]=-z+\sum\limits_{k=1}^{m+2} \delta_{i}^k e_k,\\[3mm]
[x_i,z]=\sum\limits_{k=1}^{m+2} \eta_{i}^k e_k,& [y_i,y_j]=\sum\limits_{k=1}^{m+2} \beta_{i,j}^k e_k,&
 [y_i,x_j]=\sum\limits_{k=1}^{m+2} \nu_{i,j}^k e_k,\ \ i\neq j  ,
\\[3mm]
[y_i,z]=\sum\limits_{k=1}^{m+2} \theta_{i}^k e_k,&[y_1,x_1]=z,&[y_i,x_i]=z+\sum\limits_{k=1}^{m+2} \varepsilon_{i}^k e_k,\ \ i\neq 1,\\[3mm]
[z,x_i]=\sum\limits_{k=1}^{m+2} \tau_{i}^k e_k,
&[z,y_i]=\sum\limits_{k=1}^{m+2} \lambda_{i}^k e_k,&
[z,z]=\sum\limits_{k=1}^{m+2} \mu^k e_k,
\end{array}$$
with $1\leq i,j\leq m.$

We compute all Leibniz identities  using the software Mathematica
and we get  the following restrictions:
$$\begin{array}{llll}
\mbox{ Leibniz identity }& & \mbox{ Constraint }&\\[1mm]
\hline \hline\\
\{z,z,y_k\}&\Rightarrow& \mu^{m+2}=\lambda_k^{m+2}=0,&1\leq k\leq m,\\[2mm]
\{z,z,x_k\}&\Rightarrow& \mu^{k+1}=\tau_k^{m+2}=0,&1\leq k\leq m,\\[2mm]
\{z,y_j,x_k\}&\Rightarrow&  \lambda_{j}^{k+1}=0, \ \mu^1=\lambda_1^2=\lambda_j^{j+1},&1\leq j,k\leq m,\ j\neq k,\\[2mm]
\{z,x_j,x_k\}&\Rightarrow&\tau_j^{k+1}=\tau_k^{j+1},&1\leq j,k\leq m,\\[2mm]
\{y_i,z,y_k\}&\Rightarrow& \theta_i^{m+2}=\beta_{i,k}^{m+2}=0,&1\leq i,j,k\leq m,\\[2mm]
\{y_i,z,x_k\}&\Rightarrow& \nu_{i,k}^{m+2}=\theta_{i}^{k+1},&1\leq i,k\leq m,\ i\neq k,\\[2mm]
&\Rightarrow& \theta_{i}^{i+1}-\mu^1=0,\ \mu^1=\theta_1^2,&1\leq
i\leq m,\ k= i,\\[2mm]
\{y_i,y_j,x_k\}&\Rightarrow &\beta_{i,j}^{k+1}=\nu_{i,k}^{m+2}=0, & 1 \leq i,j,k \leq m,\ j\neq k\neq i,\\[2mm]
&\Rightarrow &\theta_i^1=\beta_{i,j}^{j+1},\ \theta_i^s=0,&1\leq i,j\leq m,\ k=j,\ i\neq j,\ 2\leq s\leq m+1,\\[2mm]
&\Rightarrow& \beta_{i,j}^{i+1}=\lambda_j^1,\ -\lambda_j^{j+1}-\varepsilon_i^{m+2}=0,&1\leq i,j\leq m,\ i\neq 1,\\[2mm]
&\Rightarrow&\lambda_j^{j+1}=0,&i=1,\ j\neq 1,\\[2mm]
&\Rightarrow&\theta_1^{s}=0,&3\leq s\leq m+1,\ i=j=k=1,\\[2mm]
\{y_i,x_i,y_i\}&\Rightarrow& \theta_i^1=\beta_{i,i}^{i+1},&1\leq
i\leq m,\\[2mm]
%


\{y_i,x_j,x_k\}&\Rightarrow &\nu_{i,j}^{k+1}=\nu_{i,k}^{j+1},&1\leq i,j,k\leq m, \ j\neq i\neq k,\\[2mm]
&\Rightarrow & \tau_k^s=0,\ \tau_k^1+\varepsilon_i^{k+1}-\nu_{i,k}^{i+1}=0,&2\leq s\leq m+1,\ 1\leq i,k\leq m, \ j= i\neq k,\\[2mm]
&\Rightarrow & \tau_j^1=\nu_{1,j}^2,&1\leq j\leq m,\ i=k=1,\ j\neq 1,\\[2mm]
\{x_i,z,y_k\}&\Rightarrow& \eta_i^{m+2}=\gamma_{i,k}^{m+2}=0,&1\leq i,k\leq m,\ i\neq k,\\[2mm]
&\Rightarrow& \delta_{i}^{m+2}=0,&1\leq i\leq m, \ i=k,\\[2mm]
\{x_i,z,x_k\}&\Rightarrow& \eta_i^{k+1}=\alpha_{i,k}^{m+2},&1\leq i,k\leq m,\\[2mm]
\{x_i,y_i,y_k\}&\Rightarrow &\lambda_{k}^1=0,&1\leq k\leq m,\\[2mm]
\{x_i,y_j,y_k\}&\Rightarrow &\gamma_{i,j}^{k+1}=\alpha_{i,k}^{m+2}=0,&1\leq i,j, k\leq m,\ i\neq j\neq k\\[2mm]
&\Rightarrow &-\tau_k^1+\delta_i^{k+1}=0,&1\leq i, k\leq m,\ j=i\neq k,\\[2mm]
&\Rightarrow &\gamma_{i,j}^{i+1}=0,&1\leq ij\leq m,\ k=j\neq i,\\[2mm]
&\Rightarrow & \eta_i^1=-\tau_i^1+\delta_i^{i+1},\ \eta_{i}^s=0,&2\leq s\leq m,\ 1\leq i\leq m,\ j=k=i,\\[2mm]
\{x_i,x_j,y_k\}&\Rightarrow & \gamma_{i,j}^k=0,&1\leq i,j,k\leq m, \ i\neq k\neq j,\\[2mm]
&\Rightarrow & \eta _i^1=\gamma_{i,j}^{j+1},&1\leq i,j\leq m, \ k=j\neq i,\\[2mm]
\{x_i,x_j,x_k\}&\Rightarrow &\alpha_{i,j}^{k+1}=\alpha_{i,k}^{j+1}, & 1 \leq i,j,k \leq m.
\end{array}$$

\vspace*{0.3cm}

From here,
$$\begin{array}{ll}
[e_{i+1}, x_i] = e_1, & 1 \leq i \leq m,\\[2mm]
[e_{m+2},y_i] = e_{i+1}, & 1 \leq i \leq m,\\[2mm]
[e_{m+2}, z] = e_1,&\\[2mm]
[x_i,x_j]=\sum\limits_{s=1}^{m+1} \alpha_{i,j}^s e_s, & 1 \leq i, j \leq m,\\[2mm]
[y_i,y_j]=\beta_{i,j}^1 e_1+\theta_i^1 e_{j+1}, & 1 \leq i, j \leq m,

\end{array}$$
$$\begin{array}{ll}
[x_i,y_j]=\gamma_{i,j}^1 e_1+\eta_i^1 e_{j+1}, & 1 \leq i, j\leq m,\ i\neq j\\[2mm]

[x_1,y_1]=-z+\delta_1^1 e_1+(\eta_1^1+\tau_1^1) e_{2}+\sum\limits_{s=2}^{m} \nu_{1,s}^2 e_{s+1}, & 1 \leq i\leq m,\\[2mm]
[x_i,y_i]=-z+\delta_i^1 e_1+\tau_1^1 e_2+(\eta_i^1+\nu_{1,i}^2) e_{i+1}+\sum\limits_{s=2\, s\neq i}^{m} \nu_{1,s}^2 e_{s+1}, & 2 \leq i\leq m,\\[2mm]
[y_1,x_1]=z,& \\[2mm]
[y_i,x_i]=z+\varepsilon_i^1e_1+(\nu_{i,1}^{i+1}-\tau_1^1)e_2+\varepsilon_i^{i+1} e_{i+1}+\sum\limits_{s=2\, s\neq i}^{m} (\nu_{i,s}^{i+1}-\nu_{1,s}^2) e_{s+1}, & 2 \leq i\leq m,\\[2mm]
[y_i,x_j]=\sum\limits_{s=1}^{m+1} \nu_{i,j}^s e_s, & 1 \leq i, j  \leq m,\ i\neq j\\[2mm]
[x_i,z]=  \eta_{i}^1 e_1, & 1 \leq i \leq m,\\[2mm]
[y_i,z]= \theta_{i}^1 e_1, & 1 \leq i \leq m,\\[2mm]
[z,x_1]=  \tau_{1}^1 e_1, & \\[2mm]
[z,x_i]=  \nu_{1,i}^2 e_1, & 2 \leq i \leq m,
\end{array}$$
with the following restrictions
$$\begin{array}{ll}
\alpha_{i,j}^{k+1}=\alpha_{i,k}^{j+1}, & 1 \leq i,j,k \leq m,\\[1mm]
\nu_{i,j}^{k+1}=\nu_{i,k}^{j+1},&1\leq i,j,k\leq m, \ j\neq i\neq
k.
\end{array}$$

Only rest to make the next  change of basis
$$\left\{\begin{array}{ll}
x'_i=x_i-\eta_i^1 e_{m+2},&1\leq i\leq m,\\[1mm]
y'_1=y_1-\theta_1^1 e_{m+2},&\\[1mm]
y'_j=y_j-\varepsilon_{j}^1 e_{j+1}-\theta_j^1 e_{m+2},&2\leq j\leq m,
\end{array}\right.$$
and we obtain the family of the theorem (renaming the parameters).
\end{proof}

\subsection{Particular case: Classification of Leibniz algebras when $m=1$.}
In this subsection we classify  the Leibniz algebras such that
$L/I\cong H_3$ and $I$ is the  $L/ I$-module with the minimal
faithful representation given by Equation (\ref{final}).  Let us
fix $\{x,y,z,e_1,e_2,e_3\}$ as basis of $L.$ All computations have
been made by  using the software $Mathematica.$

We have the following products:
$$\begin{array}{lll}
[e_{2}, x ] = e_1,& [e_{3},y ] = e_{2}, & [e_{3}, z] =
e_1,\\{}
[x,x]=\alpha_{1} e_1+\alpha_{2} e_2+\alpha_{3} e_3,  &
[x,y]=-z+\delta_1 e_1+\delta_2  e_2+  \delta_3  e_{3},&
[x,z]= \eta_{1} e_1+\eta_{2} e_2+\eta_{3} e_3, \\{}
[y ,y]=\beta_{1}  e_1+\beta_{2}  e_2+\beta_{3}  e_3,&
[y,x]=z,& [y,z]= \theta_{1} e_1+\theta_{2} e_2+\theta_{3} e_3,\\{}
[z,x]=\tau_1  e_1+ \tau_2  e_2+ \tau_3  e_3,& [z,y]=  \lambda_1
e_1+\lambda_2 e_2+\lambda_3 e_3,& [z,z]=\mu_1 e_1+\mu_2
e_2+\mu_3 e_3.
\end{array}$$

The Leibniz identity on the following triples imposes further
constraints on the products.
$$\begin{array}{lll}
\mbox{ Leibniz identity }& & \mbox{ Constraint }\\[1mm]
\hline \hline\\
\{x,x,y\}&\Rightarrow &-\eta_1=\tau_1-\delta_2,\ \ \alpha_3-\eta_2=\tau_2,\ \ -\eta_3=\tau_3,\\[1mm]
\{x,x,z\}&\Rightarrow &\alpha_3=\eta_2,\\[1mm]
\{x,y,z\}&\Rightarrow &\mu_1=\delta_3,\ \ \mu_2=-\eta_3,\ \ \mu_3=0,\\[1mm]
\{y,y,z\}&\Rightarrow & \beta_3=\theta_3=0,\\[1mm]
\{y,x,y\}&\Rightarrow & -\theta_1=\lambda_1-\beta_2,\ \ -\theta_2=\lambda_2,\ \ -\theta_3=\lambda_3,\\[1mm]
\{y,x,z\}&\Rightarrow & \mu_1=\theta_2,\ \ \mu_2=0,\\[1mm]
\{z,x,y\}&\Rightarrow & \mu_1=\lambda_2,\ \ \mu_2=\tau_3,\\[1mm]
\{z,x,z\}&\Rightarrow & \mu_2=\tau_3,\\[1mm]
\{z,y,z\}&\Rightarrow & \lambda_3=0,\\[1mm]
\{x,z,x\}&\Rightarrow &\eta_2=\alpha_3,\\[1mm]
\{x,z,y\}&\Rightarrow & \mu_1=\delta_3,\ \ \mu_2=-\eta_3.
\end{array}$$

Thus, we get the following family of algebras, $L(\alpha_1,\alpha_2,\alpha_3,\beta_1,\beta_2,\delta_1,\delta_2,\eta_1,\theta_1):$
$$\left\{\begin{array}{lll}
[e_{2}, x ] = e_1,& [e_{3},y ] = e_{2}, & [e_{3}, z] =
e_1,\\{}
[x,x]=\alpha_{1} e_1+\alpha_{2} e_2+\alpha_{3} e_3,  &
[x,y]=-z+\delta_1 e_1+\delta_2  e_2,&
[x,z]= \eta_{1} e_1+\alpha_3 e_2, \\{}

[y ,y]=\beta_{1}  e_1+\beta_{2}  e_2 ,& [y,x]=z,&  [y,z]= \theta_{1} e_1,\\{}
[z,x]=  (\delta_2-\eta_1)e_1-2\alpha_3  e_2,& [z,y]=  (\beta_2-\theta_1)e_1.&
\end{array}\right.$$

\begin{teo}
Let $L$ be a Leibniz algebra such that $L/ I \cong H_{3}$ and
$I$ is the $L/ I$-module with the minimal faithful
representation given by Equation (\ref{final}). Then $L$  is
isomorphic to one of the following pairwise
non-isomorphic algebras:
$$\begin{array}{lll}
L(0,1,0,1,0,0,0,1,\lambda),\ \lambda\in \mathbb{F},&L(0,1,0,1,0,0,0,0,1),&L(0,1,0,1,0,0,0,0,0),\\[1mm]
L(0,1,0,0,0,0,0,1,\lambda),\ \lambda\in \mathbb{F},&L(0,1,0,0,0,0,0,0,1),&L(0,1,0,0,0,0,0,0,0),\\[1mm]
L(0,0,0,1,0,0,0,1,1),&L(0,0,0,1,0,0,0,1,0),&L(0,0,0,1,0,0,0,0,1),\\[1mm]
L(0,0,0,1,0,0,0,0,0), &L(0,0,0,0,0,0,0,1,1), &L(0,0,0,0,0,0,0,1,0),\\[1mm]
 L(0,0,0,0,0,0,0,0,1),&L(0,0,0,0,0,0,0,0,0),&L(0,0,1,1,0,0,0,1,\lambda),\ \lambda\in \mathbb{F},\\[1mm]
 L(0,0,1,1,0,0,0,0,1), &L(0,0,1,1,0,0,0,0,0), &L(0,0,1,0,0,0,0,1,1),\\[1mm]
 L(0,0,1,0,0,0,0,1,0), &L(0,0,1,0,0,0,0,0,1), &L(0,0,1,0,0,0,0,0,0).

\end{array}$$
\end{teo}

\begin{proof}
We can distinguish two cases:\\

\noindent {\bf Case 1:} $e_3\in [L,L].$ Then $\alpha_3=0.$\\

Applying the general change of basis generators:
$$x'=A_1x+A_2y+A_3 z+\sum_{k=1}^3 P_i e_i,\ \ y'=B_1 x+B_2y+B_3 z+\sum_{k=1}^3 Q_i e_i,\ \ e'_3=C_1 x+C_2 y+C_3z+\sum_{k=1}^3 R_i e_i$$
we derive the expressions of the new parameters in the new basis:
$$\begin{array}{ll}
\alpha'_1=\displaystyle\frac{\alpha_1 A_1^2B_2-\alpha_2
A_1^2B_3+\delta_2A_1A_3B_2+A_1B_2P_2+A_3B_2P_3}{A_1B_2^2},&
\alpha'_2=\displaystyle\frac{\alpha_2A_1^2}{B_2R_3},\\[3mm]
\beta'_1=\displaystyle\frac{\beta_1B_2}{A_1R_3},&
\beta'_2=\displaystyle\frac{\beta_2B_2+Q_3}{R_3},\\[3mm]
\delta'_1=\displaystyle\frac{\beta_2A_3B_2+\delta_1A_1B_2+A_1Q_2+A_3Q_3}{A_1B_2R_3},&
\delta'_2=\displaystyle\frac{\delta_2A_1+P_3}{R_3},\\[3mm]
\eta'_1=\displaystyle\frac{\eta_1A_1+P_3}{R_3},&
\theta'_1=\displaystyle\frac{\theta_1B_2+Q_3}{R_3},
\end{array}$$
and the following restrictions:
$$\left\{\begin{array}{l}
C_1=C_2=C_3=B_1=A_2=0,\\
R_5=-\displaystyle\frac{A_3R_3}{A_1},\\
A_1B_2R_3\neq 0.
\end{array}\right.$$

We set
$$\begin{array}{lll}
P_3=-\delta_2 A_1&\Rightarrow &\delta'_2=0,\\
Q_3=-\beta_2B_2&\Rightarrow & \beta'_2=0,\\
Q_2=-\delta_1 B_2&\Rightarrow & \delta'_1=0,\\
P_2=-\displaystyle\frac{(\alpha_1B_2-\alpha_2B_3)A_1}{B_2}&\Rightarrow
&\alpha_1'=0,
\end{array}$$
then we get
$$\begin{array}{lllll}
[e_{2}, x ] = e_1,& & [e_{3},y ] = e_{2},& &  [e_{3}, z] =
e_1,\\{}
[x,x]= \alpha'_{2} e_2, & &[x,y]=-z ,& &[x,z]= \eta'_{1} e_1, \\{}
[y ,y]=\beta'_{1}  e_1
,&  &[y,x]=z,& &[y,z]= \theta'_{1} e_1,\\{}
[z,x]=  -\eta'_1  e_1,& &[z,y]=
-\theta'_1e_1,& &
\end{array}$$
where
$$\begin{array}{llll}
\alpha'_2=\displaystyle\frac{\alpha_2A_1^2}{B_2R_3},&
\beta'_1=\displaystyle\frac{\beta_1B_2}{A_1R_3},&
\eta'_1=\displaystyle\frac{(\eta_1-\delta_2)A_1}{R_3},&
\theta'_1=\displaystyle\frac{(\theta_1-\beta_2)B_2}{R_3}.
\end{array}$$

We observe that the nullities of $\alpha_2,\ \beta_1,\ \eta_1,\
\theta_1$ are invariant. Thus, we can distinguish the following
non-isomorphic cases. An appropriate choice of the parameter
values ($A_1,$ $B_2$ and $R_3$)  allows us to obtain the following
algebras or families of algebras.
$$\begin{array}{|l|l|}
\hline
\mbox{Case}&\mbox{Algebra}\\[2mm]
\hline \hline
\alpha_2\neq 0,\ \beta_1\neq 0,\ \eta_1\neq 0,&L(0,1,0,1,0,0,0,1,\lambda),\ \lambda\in \mathbb{F},\\[1mm]
\alpha_2\neq 0,\ \beta_1\neq 0,\ \eta_1= 0,\ \theta_1\neq 0,&L(0,1,0,1,0,0,0,0,1),\\[1mm]
\alpha_2\neq 0,\ \beta_1\neq 0,\ \eta_1= 0,\ \theta_1= 0,&L(0,1,0,1,0,0,0,0,0),\\[1mm]
\alpha_2\neq 0,\ \beta_1=0,\ \eta_1\neq  0,&L(0,1,0,0,0,0,0,1,\lambda),\ \lambda\in \mathbb{F},\\[1mm]
\alpha_2\neq 0,\ \beta_1=0,\ \eta_1= 0,\ \theta_1\neq 0,&L(0,1,0,0,0,0,0,0,1),\\[1mm]
\alpha_2\neq 0,\ \beta_1=0,\ \eta_1= 0,\ \theta_1= 0,&L(0,1,0,0,0,0,0,0,0),\\[1mm]
\alpha_2= 0,\ \beta_1\neq 0,\ \eta_1\neq 0,\ \theta_1\neq 0,&L(0,0,0,1,0,0,0,1,1),\\[1mm]
\alpha_2= 0,\ \beta_1\neq 0,\ \eta_1\neq 0,\ \theta_1=0,&L(0,0,0,1,0,0,0,1,0),\\[1mm]
\alpha_2= 0,\ \beta_1\neq 0,\ \eta_1= 0,\ \theta_1\neq 0,&L(0,0,0,1,0,0,0,0,1),\\[1mm]
\alpha_2= 0,\ \beta_1\neq 0,\ \eta_1= 0,\ \theta_1= 0,&L(0,0,0,1,0,0,0,0,0),\\[1mm]
\alpha_2= 0,\ \beta_1= 0,\ \eta_1\neq 0,\ \theta_1\neq 0,&L(0,0,0,0,0,0,0,1,1),\\[1mm]
\alpha_2= 0,\ \beta_1= 0,\ \eta_1\neq 0,\ \theta_1= 0,&L(0,0,0,0,0,0,0,1,0),\\[1mm]
\alpha_2= 0,\ \beta_1= 0,\ \eta_1= 0,\ \theta_1\neq 0,&L(0,0,0,0,0,0,0,0,1),\\[1mm]
\alpha_2= 0,\ \beta_1= 0,\ \eta_1= 0,\ \theta_1= 0,&L(0,0,0,0,0,0,0,0,0).\\[1mm]
\hline
\end{array}$$

\noindent {\bf Case 2:} $e_3\notin [L,L].$ Then $\alpha_3\neq 0.$
Making the following change of basis in
$L(\alpha_1,\alpha_2,\alpha_3,\beta_1,\beta_2,\delta_1,\delta_2,\eta_1,\theta_1)$
$$\left\{\begin{array}{l}
e'_3=\alpha_1 e_1+\alpha_2e_2+\alpha_3e_3,\\{}
e'_2=\alpha_3e_2,\\{} e'_1=\alpha_3e_1,
\end{array}\right.$$
we obtain $L(0,0,1,\beta_1,\beta_2,\delta_1,\delta_2,\eta_1,\theta_1):$
$$\left\{\begin{array}{lll}
[e_{2}, x ] = e_1,& [e_{3},y ] = e_{2}, & [e_{3}, z] =
e_1,\\{}
[x,x]= e_3, & [x,y]=-z+\delta_1 e_1+\delta_2  e_2,&[x,z]= \eta_{1} e_1+ e_2, \\{}
[y ,y]=\beta_{1}  e_1+\beta_{2}  e_2,& [y,x]=z,& [y,z]= \theta_{1} e_1,\\{}
[z,x]=(\delta_2-\eta_1)  e_1-2   e_2,& [z,y]=  (\beta_2-\theta_1)e_1.&
\end{array}\right.$$

Analogously to the previous case, by making the general change of
basis of generators
$$x'=A_1x+A_2y+A_3 z+\sum_{k=1}^3 P_i e_i,\ \qquad y'=B_1 x+B_2y+B_3 z+\sum_{k=1}^3 Q_i e_i,$$
we derive the expressions of the new parameters in the new basis:
$$\begin{array}{ll}
\beta_1'=\displaystyle\frac{\beta_1B_2}{A_1^3},&
\beta_2'=\displaystyle\frac{\beta_2B_2+Q_3}{A_1^2},\\[2mm]
\delta_1'=\displaystyle\frac{\beta_2A_3B_2^2+A_1B_3^2+\delta_1A_1B_2^2+A_1B_2Q_2+A_3B_2Q_3}{A_1^3B_2^2},&
\delta'_2=\displaystyle\frac{-A_1B_3+\delta_2A_1B_2+B_2P_3}{A_1^2B_2},\\[2mm]
\eta'_1=\displaystyle\frac{-A_1B_3+\eta_1
A_1B_2+B_2P_3}{A_1^2B_2},& \theta'_1=\displaystyle\frac{\theta_1
B_2+Q_3}{A_1^2},
\end{array}$$
with the restriction:
$$\left\{\begin{array}{l}
A_2=B_1=0,\\
A_1B_2\neq 0.
\end{array}\right.$$

By putting
$$\begin{array}{lll}
P_3=\displaystyle\frac{A_1(B_3-\delta_2B_2)}{B_2}&\Rightarrow &\delta'_2=0,\\
Q_3=-\beta_2B_2&\Rightarrow & \beta'_2=0,\\
Q_2=-\displaystyle\frac{B_3^2+\delta_1B_2^2}{B_2}&\Rightarrow
&\delta_1'=0,
\end{array}$$
we deduce
$$ \begin{array}{lllll}
[e_{2}, x ] = e_1,& &[e_{3},y ] = e_{2}, & &[e_{3}, z] =
e_1,\\{}
[x,x]= e_3,&  &[x,y]=-z ,& &[x,z]= \eta'_{1} e_1+e_2, \\{}
[y ,y]=\beta'_{1}  e_1 ,&  &[y,x]=z,&  &[y,z]=
\theta'_{1} e_1,\\{}
[z,x]=  -\eta'_1  e_1-2e_2,& &[z,y]=
-\theta'_1e_1,& &
\end{array}$$
where
$$\begin{array}{lll}
\beta'_1=\displaystyle\frac{\beta_1B_2}{A_1^3},&
\eta'_1=\displaystyle\frac{\eta_1-\delta_2}{A_1},&
\theta'_1=\displaystyle\frac{(\theta_1-\beta_2)B_2}{A_1^2}.
\end{array}$$

We observe that the nullities of $\beta_1,\ \eta_1,\ \theta_1$ are
invariant. Thus, we can distinguish the following non-isomorphic cases. An appropriate choice of the parameter values
($A_1$ and $B_2$)  allows us to obtain the following algebras or
families of algebras.
$$\begin{array}{|l|l|}
\hline
\mbox{Case}&\mbox{Algebra}\\[2mm]
\hline \hline
 \beta_1\neq 0,\ \eta_1\neq 0,&L(0,0,1,1,0,0,0,1,\lambda),\ \lambda\in \mathbb{F},\\[1mm]
 \beta_1\neq 0,\ \eta_1= 0,\ \theta\neq 0,&L(0,0,1,1,0,0,0,0,1),\\[1mm]
 \beta_1\neq 0,\ \eta_1= 0,\ \theta= 0,&L(0,0,1,1,0,0,0,0,0),\\[1mm]
 \beta_1= 0,\ \eta_1\neq  0,\ \theta\neq 0,&L(0,0,1,0,0,0,0,1,1),\\[1mm]
 \beta_1= 0,\ \eta_1\neq  0,\ \theta= 0,&L(0,0,1,0,0,0,0,1,0),\\[1mm]
 \beta_1= 0,\ \eta_1=  0,\ \theta\neq 0,&L(0,0,1,0,0,0,0,0,1),\\[1mm]
 \beta_1= 0,\ \eta_1=  0,\ \theta=0,&L(0,0,1,0,0,0,0,0,0).\\[1mm]
\hline
\end{array}$$

The proof is complete.

\end{proof}


\

{\sc Antonio J. Calder\'{o}n.}  Dpto. Matem\'{a}ticas. Universidad
de C\'{a}diz. 11510 Puerto Real, C\'{a}diz. (Spain), e-mail:
\emph{ajesus.calderon@uca.es}

\

{\sc Luisa M. Camacho.}  Dpto. Matem\'{a}tica Aplicada I.
Universidad de Sevilla. Avda. Reina Mercedes, s/n. 41012 Sevilla.
(Spain), e-mail: \emph{lcamacho@us.es}

\

{\sc Bakhrom A. Omirov.} Institute of Mathematics. National
University of Uzbekistan, F. Hodjaev str. 29, 100125, Tashkent
(Uzbekistan), e-mail: \emph{omirovb@mail.ru}

\end{document}